\documentclass[letterpaper,10 pt,journal,twoside]{IEEEtran}
\usepackage[utf8]{inputenc}
\usepackage{amsmath}
\usepackage{amsfonts}
\usepackage{amssymb}

\usepackage{amsthm}
\usepackage{color}
\usepackage{mathtools}
\usepackage{graphicx}
\usepackage[caption=false, font=footnotesize]{subfig}
\newtheorem{theorem}{Theorem}   
\newtheorem{lemma}[theorem]{Lemma}
\newtheorem{problem}[]{Problem}

\newtheoremstyle{example}
{3pt}
{3pt}
{}
{}
{\upshape \bfseries}
{.}
{.5em}
{}
\theoremstyle{example}
\newtheorem{example}{Example}
\newtheorem{remark}{Remark}
\newtheorem{observation}{Observation}
\newtheorem{dfn}{Definition}

\DeclareMathOperator{\col}{col}

\DeclareMathOperator{\diam}{diam}
                     
\IEEEoverridecommandlockouts                              

\title{Identifiability of Undirected Dynamical Networks: a Graph-Theoretic Approach}

\author{Henk J. van Waarde, Pietro Tesi, and M. Kanat Camlibel 
\thanks{Pietro Tesi and Henk van Waarde are with the Engineering and Technology Institute Groningen, Faculty of Science and Engineering, University of Groningen, 9747 AG Groningen, The Netherlands. Pietro Tesi is also with the Department of Information Engineering, University of Florence, 50139 Florence, Italy. Kanat Camlibel and Henk van Waarde are with the Johann Bernoulli Institute for Mathematics and Computer Science, Faculty of Science and Engineering, University of Groningen, P.O. Box 800, 9700 AV Groningen, The Netherlands. Email: {\tt\small h.j.van.waarde@rug.nl, p.tesi@rug.nl, pietro.tesi@unifi.it, m.k.camlibel@rug.nl}.
}%
}

\begin{document}

\maketitle
\thispagestyle{empty}

\begin{abstract}
This paper deals with identifiability of undirected dynamical networks with single-integrator node dynamics. We assume that the graph structure of such networks is known, and aim to find graph-theoretic conditions under which the state matrix of the network can be uniquely identified. As our main contribution, we present a graph coloring condition that ensures identifiability of the network's state matrix. Additionally, we show how the framework can be used to assess identifiability of dynamical networks with general, higher-order node dynamics. As an interesting corollary of our results, we find that excitation and measurement of all network nodes is not required. In fact, for many network structures, identification is possible with only small fractions of measured and excited nodes.
\end{abstract}


\section{Introduction}
\IEEEPARstart{N}{etworks} of dynamical systems appear in multiple contexts, including power networks, sensor networks, and robotic networks (cf. Section 1 of \cite{Mesbahi2010}). It is natural to describe such networks by a graph, where nodes correspond with dynamical subsystems, and edges represent interaction between different systems. Often, the graph structure of dynamical networks is not directly available. For instance, in neuroscience, the interactions between brain areas are typically unknown \cite{Sosa2005}. Other examples of networks with unknown interconnection structure include genetic networks \cite{Julius2009} and wireless sensor networks \cite{Mao2007}.

Consequently, the problem of \emph{network reconstruction} is considered in the literature. Network reconstruction is quite a broad concept, and there exist multiple variants of this problem. For example, the goal in \cite{Shahrampour2015}, \cite{Materassi2012} is to reconstruct the Boolean structure of the network (i.e., the locations of the edges). Moreover, simultaneous identification of the graph structure \emph{and} the network weights has been considered in \cite{Hassan-Moghaddam2016}, \cite{vanWaarde2018}, \cite{Zhou2007}. Typically, the conditions under which the network structure is uniquely identifiable are rather strong, and it is often assumed that the states of all nodes in the network can be measured \cite{Materassi2012}, \cite{Hassan-Moghaddam2016}, \cite{vanWaarde2018}, \cite{Zhou2007}. In fact, it has been shown \cite{Pare2013} that measuring all network nodes is \emph{necessary} for network reconstruction of dynamical networks (described by a class of state-space systems).  

In this paper, we consider undirected dynamical networks described by state-space systems. In contrast to the above discussed papers, we assume that the graph structure is \emph{known}, but the state matrix of the network is \emph{unavailable}. Such a situation arises, for example, in electrical or power networks in which the locations of links are typically known, but link weights require identification. Our goal is to find graph-theoretic conditions under which the state matrix of the network can be uniquely identified. 

Graph-theoretic conditions have previously been used to assess other system-theoretic properties such as structural controllability \cite{Monshizadeh2014}, \cite{vanWaarde2017}, fault detection \cite{Rapisarda2015}, \cite{deRoo2015}, and parameter-independent stability \cite{Koerts2017}. Conditions based on the graph structure are often desirable since they avoid potential numerical issues associated with more traditional linear algebra tests. In addition, graph-theoretic conditions provide insight in the types of networks having certain system-theoretic properties, and can aid the selection of input/output nodes \cite{Pequito2016}.

The papers that are most closely related to the work in this paper are \cite{Nabavi2016} and \cite{Bazanella2017}. Nabavi \emph{et al.} \cite{Nabavi2016} consider weighted, undirected consensus networks with a single input. They assume that the graph structure is known, and aim to identify the weights in the network. A sensor placement algorithm is presented, which selects a set of sensor nodes on the basis of the graph structure. It is shown that this set of sensor nodes is sufficient to guarantee weight identifiability. Bazanella \emph{et al.} \cite{Bazanella2017} consider a network model where interactions between nodes are modeled by proper transfer functions (see also \cite{vandenHof2013}, \cite{Weerts2018}). Also in this paper, the graph structure is assumed to be known, and the goal is to find conditions under which the transfer functions can be identified. Under the assumption that all nodes are externally excited, necessary and sufficient graph-theoretic conditions are presented under which all transfer functions can be (generically) identified. 

Note that the above papers make explicit assumptions on the number of input or output nodes. Indeed, in \cite{Nabavi2016} there is a single input node, all nodes are input nodes in \cite{Bazanella2017}, and all nodes are measured in \cite{Weerts2018}. In contrast to these papers, the current paper considers graph-theoretic conditions for identifiability of dynamical networks where the sets of input and output nodes can be any two (known) subsets of the vertex set. Our main contribution consists of a graph coloring condition for identifiability of dynamical networks with single-integrator node dynamics. Specifically, we prove a relation between identifiability and so-called \emph{zero forcing sets} \cite{Hogben2010} (see also \cite{Monshizadeh2014}, \cite{vanWaarde2017}, \cite{Trefois2015}). As our second result, we show how our framework can be used to assess identifiability of dynamical networks with general, higher-order node dynamics. 

The organization of this paper is as follows. First, in Section \ref{sectionpreliminaries} we introduce the notation and preliminaries used throughout the paper. Subsequently, in Section \ref{sectionproblem} we state the problem. Section \ref{sectionmainresult} contains our main results, and Section \ref{sectionextension} treats an extension to higher-order dynamics. Finally, our conclusions are stated in Section \ref{sectionconclusion}.

\section{Preliminaries}
\label{sectionpreliminaries}
We denote the sets of natural, real, and complex numbers by $\mathbb{N}$, $\mathbb{R}$, and $\mathbb{C}$, respectively. Moreover, the set of real $m \times n$ matrices is denoted by $\mathbb{R}^{m \times n}$ and the set of symmetric $n \times n$ matrices is given by $\mathbb{S}^n$. The \emph{transpose} of a matrix $A$ is denoted by $A^\top$. A \emph{principal submatrix} of $A \in \mathbb{R}^{n \times n}$ is a square submatrix of $A$ obtained by removing rows and columns from $A$ with the same indices. We denote the Kronecker product of two matrices $A$ and $B$ by $A \otimes B$. The $n \times n$ identity matrix is given by $I_n$. If the dimension of $I_n$ is clear from the context, we simply write $I$. For $x_1,x_2,\dots,x_n \in \mathbb{R}^q$, we use the notation $\col(x_1,x_2,\dots,x_n) \in \mathbb{R}^{qn}$ to denote the concatenation of the vectors $x_1,x_2,\dots,x_n$. Finally, the cardinality of a set $S$ is denoted by $|S|$.

\subsection{Graph theory}
\label{sectiongraphtheory}
All graphs considered in this paper are simple, that is, without self-loops and with at most one edge between any pair of vertices. Let $G = (V,E)$ be an undirected graph, where $V = \{1,2,\dots,n\}$ is the set of \emph{vertices} (or nodes), and $E \subseteq V \times V$ denotes the set of \emph{edges}. A node $j \in V$ is said to be a \emph{neighbour} of $i \in V$ if $(i,j) \in E$. An \emph{induced subgraph} $G_S = (V_S,E_S)$ of $G$ is a graph with the properties that $V_S \subseteq V$, $E_S \subseteq E$ and for each $i,j \in V_S$ we have $(i,j) \in E_S$ if and only if $(i,j) \in E$. For any subset of nodes $V' = \{v_1,v_2,\dots,v_r\} \subseteq V$ we define the $n \times r$ matrix $P(V;V')$ as 
$P_{ij} := 1$ if $i = v_j$ and $P_{ij} := 0$ otherwise, where $P_{ij}$ denotes the $(i,j)$-th entry of $P$. We will now define two families of matrices associated with the graph $G$. Firstly, we define the \emph{qualitative class} $\mathcal{Q}(G)$ as \cite{Hogben2010}
\begin{equation*}
\mathcal{Q}(G) := \{X \in \mathbb{S}^n \mid \text{for } i \neq j, \: X_{ji} \neq 0 \iff (i,j) \in E \}.
\end{equation*}
The off-diagonal entries of matrices in $\mathcal{Q}(G)$ carry the graph structure of $G$ in the sense that $X_{ji}$ is nonzero if and only if there exists an edge $(i,j)$ in the graph $G$. Note that the diagonal elements of matrices in $\mathcal{Q}(G)$ are free, and hence, both Laplacian and adjacency matrices associated with $G$ are contained in $\mathcal{Q}(G)$ (see \cite{Monshizadeh2014}). In this paper, we focus on a subclass of $\mathcal{Q}(G)$, namely the class of matrices with \emph{non-negative} off-diagonal entries. This class is denoted by $\mathcal{Q}_p(G)$, and defined as
\begin{equation*}
\mathcal{Q}_p(G) := \{X \in \mathcal{Q}(G) \mid \text{for } i \neq j, \: X_{ji} \neq 0 \implies X_{ji} > 0 \}.
\end{equation*}
Note that (weighted) adjacency and negated Laplacian matrices are members of the class $\mathcal{Q}_p(G)$.

\begin{remark}
In this paper, we focus on undirected \emph{loopless} graphs $G$, and on the associated class of matrices $\mathcal{Q}_p(G)$. One could also define a class of matrices $\mathcal{Q}^l_p(G_l)$ for a graph $G_l$ \emph{with} self-loops, where a diagonal entry of a matrix in $\mathcal{Q}^l_p(G_l)$ is nonzero if and only if there is a self-loop on the corresponding node in $G_l$ (see, e.g., \cite{Hogben2010}). However, since diagonal entries of matrices in $\mathcal{Q}_p(G)$ are completely \emph{free}, we obtain $\mathcal{Q}^l_p(G_l) \subseteq \mathcal{Q}_p(G)$, where $G$ is the graph obtained from $G_l$ by removing its self-loops. As a consequence, all results in this paper are also valid for graphs with self-loops.
\end{remark}

\subsection{Zero forcing sets}
\label{sectionzeroforcing}
In this section we review the notion of zero forcing. Let $G = (V,E)$ be an undirected graph with vertices colored either black or white. The \emph{color-change rule} is defined in the following way. If $u \in V$ is a black vertex and exactly one neighbour $v \in V$ of $u$ is white, then change the color of $v$ to black \cite{Hogben2010}. When the color-change rule is applied to $u$ to change the color of $v$, we say $u$ \emph{forces} $v$, and write $u \to v$.
Given a coloring of $G$, that is, given a set $Z \subseteq V$ containing black vertices only, and a set $V \setminus Z$ consisting of only white vertices, the \emph{derived set} $D(Z)$ is the set of black vertices obtained by applying the color-change rule until no more changes are possible \cite{Hogben2010}. A \emph{zero forcing set} for $G$ is a subset of vertices $Z \subseteq V$ such that if initially the vertices in $Z$ are colored black and the remaining vertices are colored white, then $D(Z) = V$. Finally, a zero forcing set $Z \subseteq V$ is called a \emph{minimum zero forcing set} if for any zero forcing set $Y$ in $G$ we have $|Y| \geq |Z|$. 

\subsection{Dynamical networks}
Consider an undirected graph $G = (V,E)$. Let $V_I \subseteq V$ be the set of so-called \emph{input nodes}, and let $V_O \subseteq V$ be the set of \emph{output nodes}, with cardinalities $|V_I| = m$ and $|V_O| = p$, respectively. Associated with $G$, $V_I$, and $V_O$, we consider the dynamical system 
\begin{subequations}
\label{system}
\begin{align}
\dot{x}(t) &= X x(t) + M u(t) \label{system:a}\\
y(t) &= N x(t), \label{system:b}
\end{align}
\end{subequations}
where $x \in \mathbb{R}^n$ is the state, $u \in \mathbb{R}^m$ is the input, and $y \in \mathbb{R}^p$ is the output. Furthermore, $X \in \mathcal{Q}_p(G)$ and the matrices $M$ and $N$ are indexed by $V_I$ and $V_O$, respectively, in the sense that
\begin{equation}
\label{definitionMN}
M = P(V;V_I), \text{ and } N = P^\top(V;V_O).
\end{equation}
We use the shorthand notation $(X,M,N)$ to denote the dynamical system \eqref{system}. The transfer matrix of \eqref{system} is given by $T(s) := N(sI - X)^{-1} M$, where $s \in \mathbb{C}$ is a complex variable.
\begin{remark}
	Note that in this paper, we focus on dynamical networks \eqref{system}, where the state matrix $X$ is contained in $\mathcal{Q}_p(G)$. This implies that $X$ is symmetric and the off-diagonal elements of $X$ are non-negative. Dynamical networks of this form appear, for example, in consensus problems \cite{Saber2007}, and in the study of resistive-capacitive electrical networks (cf. Section V-B of \cite{Dorfler2018}). In addition, as we will see, the constraints on the matrix $X$ are also attractive from identification point of view in the sense that we can often identify $X$ with relatively small sets of input and output nodes. This is in contrast to the case of identifiability of matrices that do not satisfy symmetry and/or sign constraints. This is explained in more detail in Remark \ref{remarkdirected}.
\end{remark}

\subsection{Network identifiability}

In this section, we define the notion of network identifiability. It is well-known that the transfer matrix from $u$ to $y$ of system \eqref{system} can be identified from measurements of $u(t)$ and $y(t)$ if the input function $u$ is sufficiently rich \cite{Ljung1999}. Then, the question is whether we can uniquely reconstruct the state matrix $X$ from the transfer matrix $T(s)$. Specifically, since we assume that the matrix $X$ is \emph{unknown}, we are interested in conditions under which $X$ can be reconstructed from $T(s)$ \emph{for all} matrices $X \in \mathcal{Q}_p(G)$. This is known as \emph{global identifiability} (see, e.g., \cite{Grewal1976}). To be precise, we have the following definition.   

\begin{dfn}
	\label{definitionid}
	Consider an undirected graph $G = (V,E)$ with input nodes $V_I \subseteq V$ and output nodes $V_O \subseteq V$. Define $M$ and $N$ as in \eqref{definitionMN}. We say $(G;V_I;V_O)$ is \emph{identifiable} if for all matrices $X, \bar{X} \in \mathcal{Q}_p(G)$ the following implication holds:
	\begin{equation}
	\label{implication}
	N(sI-X)^{-1}M = N(sI-\bar{X})^{-1}M \implies X = \bar{X}.
	\end{equation}
\end{dfn}
Note that identifiability of $(G;V_I;V_O)$ is a property of the graph and the input/output nodes only, and not of the particular state matrix $X \in \mathcal{Q}_p(G)$.
\begin{observation}
The implication \eqref{implication} that appears in Definition \ref{definitionid} can be equivalently stated as
\begin{equation*}
N X^k M = N \bar{X}^k M \text{ for all } k \in \mathbb{N} \implies X = \bar{X}.
\end{equation*}
The matrices $N X^k M$ for $k \in \mathbb{N}$ are often referred to as the \emph{Markov parameters} of $(X,M,N)$.
\end{observation}

In addition to identifiability of $(G;V_I;V_O)$, we are interested in a more general property, namely identifiability of an \emph{induced subgraph} of $G$. This is defined as follows. 

\begin{dfn}
	\label{definitionid2}
	Consider an undirected graph $G = (V,E)$ with input nodes $V_I \subseteq V$ and output nodes $V_O \subseteq V$, and let $G_S$ be an induced subgraph of $G$. Define $M$ and $N$ as in \eqref{definitionMN}. We say $(G_S;V_I;V_O)$ is \emph{identifiable} if for all matrices $X, \bar{X} \in \mathcal{Q}_p(G)$ the following implication holds:
	\begin{equation*}
	N(sI-X)^{-1}M = N(sI-\bar{X})^{-1}M \implies X_S = \bar{X}_S,
	\end{equation*}
	where $X_{S},\bar{X}_{S} \in \mathcal{Q}_p(G_S)$ are the principal submatrices of $X$ and $\bar{X}$ corresponding to the nodes in $G_S$.
\end{dfn}

Note that identifiability of $(G;V_I;V_O)$ is a special case of identifiability of $(G_S;V_I;V_O)$, where the subgraph $G_S$ is simply equal to $G$. 

\section{Problem statement}
\label{sectionproblem}
Let $G = (V,E)$ be an undirected graph with input nodes $V_I \subseteq V$ and output nodes $V_O \subseteq V$, and consider the associated dynamical system \eqref{system}. Throughout this paper, we assume $G$, $V_I$, and $V_O$ to be \emph{known}. We want to investigate which principal submatrices of $X$ can be identified from input/output data (for all $X \in \mathcal{Q}_p(G)$). In other words, we want to find out for which induced subgraphs $G_S$ of $G$, the triple $(G_S;V_I;V_O)$ is identifiable. In particular, we are interested in conditions under which $(G;V_I;V_O)$ is identifiable.

Note that it is not straightforward to check the condition for identifiability in Definitions \ref{definitionid} and \ref{definitionid2}. Indeed, these definitions requires the computation and comparison of an infinite number of transfer matrices (for all $X,\bar{X}\in \mathcal{Q}_p(G)$). Instead, in this paper we want to establish a condition for identifiability of $(G_S;V_I;V_O)$ in terms of zero forcing sets. Such a graph-based condition has the potential of being more efficient to check than the condition of Definition \ref{definitionid2}. In addition, graph-theoretic conditions have the advantage of avoiding possible numerical errors in the linear algebra computations appearing in Definition \ref{definitionid2}. Explicitly, the considered problem in this paper is as follows.
\begin{problem}
	Consider an undirected graph $G = (V,E)$ with input nodes $V_I \subseteq V$ and ouput nodes $V_O \subseteq V$, and let $G_S$ be an induced subgraph of $G$. Provide graph-theoretic conditions under which $(G_S;V_I;V_O)$ is identifiable.
\end{problem}

\section{Main results}
\label{sectionmainresult}
In this section, we state our main results. First, we establish a lemma which will be used to prove our main contributions (Theorems \ref{mainresult} and \ref{mainresult2}). The following lemma considers the case that $V_I = V_O$, and asserts that identifiability of a triple $(G_S;U;U)$ is invariant under the color-change rule.
\begin{lemma}
\label{lemmainvariant}
Let $G_S$ be an induced subgraph of the undirected graph $G = (V,E)$, and let $U \subseteq V$. Suppose that $u \to v$, where $u \in U$ and $v \in V \setminus U$. Then $(G_S;U;U)$ is identifiable if and only if $(G_S;U \cup \{v\};U \cup \{v\})$ is identifiable.
\end{lemma}
\begin{proof}
The `only if' part of the statement follows directly from the fact that identifiability is preserved under the addition of input and output nodes. Therefore, in what follows, we focus on proving the `if' part. Suppose that $(G_S;U \cup \{v\};U \cup \{v\})$ is identifiable. Let $\bar{M} := P(V;U \cup \{v\})$ denote the associated input matrix, and let $\bar{N} := \bar{M}^\top$ be the output matrix. In addition, let $M := P(V;U)$ and $N := M^\top$. The idea of this proof is as follows. For any $X \in \mathcal{Q}_p(G)$, we will show that the Markov parameters $\bar{N}X^k\bar{M}$ for $k \in \mathbb{N}$ can be obtained from the Markov parameters
\begin{equation}
\label{MarkovNXM}
NX^k M \text{ for } k \in \mathbb{N}.
\end{equation}
Then, we will show that this implies that $(G_S;U;U)$ is identifiable. In particular, due to the overlap in the Markov parameters of $(\bar{N},X,\bar{M})$ and $(N,X,M)$, we only need to show that $(X^k)_{vw} = (X^k)_{wv}$ and $(X^k)_{vv}$ can be obtained from \eqref{MarkovNXM} for all $k \in \mathbb{N}$ and all $w \in U$. We start by showing that $X_{uv}$ can be obtained from \eqref{MarkovNXM}. To this end, we define $V_u := \{u\} \cup \{j \in V \mid (u,j) \in E\}$ and compute
\begin{equation*}
\begin{aligned}
(X^2)_{uu}  &= \sum_{z \in V_u} X_{uz} X_{zu} \\
			&= X_{uv}^2 + \sum_{z \in V_u \setminus \{v\}} X_{uz} X_{zu}.
\end{aligned}
\end{equation*}
By hypothesis, $u \to v$ and hence $V_u \setminus \{v\} \subseteq U$. This implies that $X_{uv}^2 = (X^2)_{uu} - \sum_{z \in V_u \setminus \{v\}} X_{uz} X_{zu}$ can be obtained from the Markov parameters \eqref{MarkovNXM}. As $X \in \mathcal{Q}_p(G)$, we have $X_{uv} > 0$ and therefore also $X_{uv}$ can be obtained from \eqref{MarkovNXM}.

Next, we prove that $(X^k)_{vw}$ can be obtained from \eqref{MarkovNXM} for any $k \in \mathbb{N}$ and any $w \in U$. To this end, we write
\begin{equation*}
\begin{aligned}
(X^{k+1})_{uw} = X_{uv} (X^k)_{vw} + \sum_{z \in V_u \setminus \{v\}} X_{uz} (X^k)_{zw}.
\end{aligned}
\end{equation*}
Since $V_u \setminus \{v\} \subseteq U$, and $X_{uv}$ can be obtained from \eqref{MarkovNXM}, this shows that we can find $(X^k)_{vw}$ from the Markov parameters \eqref{MarkovNXM} using the formula 
\begin{equation*}
(X^k)_{vw} = \frac{1}{X_{uv}} \left( (X^{k+1})_{uw} - \sum_{z \in V_u \setminus \{v\}} X_{uz} (X^k)_{zw} \right).
\end{equation*}
Finally, we have to show that $(X^k)_{vv}$ can be obtained from \eqref{MarkovNXM} for any $k \in \mathbb{N}$. To do so, we compute
\begin{equation*}
\begin{aligned}
(X^{k+2})_{uu}  &= \sum_{i,j \in V_u} X_{ui} (X^k)_{ij} X_{ju} \\
				&= X^2_{uv} (X^k)_{vv} + \sum_{\substack{i,j \in V_u\\\{i,j\} \neq \{v\}}} X_{ui} (X^k)_{ij} X_{ju}.
\end{aligned}
\end{equation*}
Note that $(X^{k+2})_{uu}$ appears as an entry of one of the Markov parameters \eqref{MarkovNXM}. Furthermore, we have already established that $X_{uv}$ can be obtained from \eqref{MarkovNXM}. If $i = v$, then $X_{ui} = X_{uv}$, and we obtain $X_{ui}$ from \eqref{MarkovNXM}. Otherwise, $i \in V_u \setminus \{v\}$, and $X_{ui}$ already appears as an entry of one of the Markov parameters \eqref{MarkovNXM}. We can repeat the exact same argument for $X_{ju}$, to show that it can be obtained from \eqref{MarkovNXM}. Finally, consider the term $(X^k)_{ij}$ for $i$ and $j$ not both equal to $v$. If $i,j \in V_u \setminus \{v\}$, then $i,j \in U$ and $(X^k)_{ij}$ appears directly as an entry of a Markov parameter in \eqref{MarkovNXM}. Next, if $i = v$, then $j \in U$ and we have already proven that $(X^k)_{vj}$ can be obtained from \eqref{MarkovNXM}. By symmetry, this also holds for the entry $(X^k)_{iv}$, where $i \in U$. This shows that $(X^k)_{vv}$ can be found using the Markov parameters \eqref{MarkovNXM} via the fomula 
\begin{equation*}
(X^k)_{vv} = \frac{1}{X^2_{uv}} \left( (X^{k+2})_{uu} - \sum_{\substack{i,j \in V_u\\\{i,j\} \neq \{v\}}} X_{ui} (X^k)_{ij} X_{ju} \right). 
\end{equation*}
Now, by hypothesis, for any $X,\bar{X}\in \mathcal{Q}_p(G)$ the following implication holds:
\begin{equation}
\label{MarkovbarNXM}
\bar{N}X^k \bar{M} = \bar{N}\bar{X}^k \bar{M} \text{ for all } k \in \mathbb{N} \implies X_S = \bar{X}_S,
\end{equation}
where $X_S$ and $\bar{X}_S$ are the principal submatrices of respectively $X$ and $\bar{X}$ corresponding to the nodes in $G_S$. Suppose that $NX^k M = N \bar{X}^k M$ for all $k \in \mathbb{N}$. By the above formulae for $(X^k)_{vv}$ and $(X^k)_{vw}$ (and for $(\bar{X}^k)_{vv}$, $(\bar{X}^k)_{vw}$), we conclude that $\bar{N}X^k \bar{M} = \bar{N}\bar{X}^k \bar{M}$ for all $k \in \mathbb{N}$, and consequently $X_S = \bar{X}_S$ by \eqref{MarkovbarNXM}. Therefore, $(G_S;U;U)$ is identifiable.
\end{proof}

Based on the previous lemma, we state the following theorem, which is one of the main results of this paper. Loosely speaking, it states that we can identify the principal submatrix of $X$ corresponding to the \emph{derived set} (cf. Section \ref{sectionzeroforcing}) of $V_I \cap V_O$.
\begin{theorem}
\label{mainresult}
Let $G_S = (V_S,E_S)$ be an induced subgraph of the undirected graph $G = (V,E)$, and let $V_I, V_O \subseteq V$. Define $W := V_I \cap V_O$ and let $D(W)$ be the derived set of $W$ in $G$. If $V_S \subseteq D(W)$ then $(G_S;V_I;V_O)$ is identifiable.
\end{theorem}
\noindent
\begin{proof}
Let $G_W$ denote the induced subgraph of $G$ with vertex set $D(W)$. Note that $(G_W;D(W);D(W))$ is trivially identifiable. By consecutive application of Lemma \ref{lemmainvariant}, we find that $(G_W;W;W)$ is identifiable. By hypothesis, $G_S$ is a subgraph of $G_W$ and hence $(G_S;W;W)$ is identifiable. Finally, note that $W \subseteq V_I$ and $W \subseteq V_O$. Since identifiability is invariant under the addition of input/output nodes, we conclude that $(G_S;V_I;V_O)$ is identifiable.
\end{proof}

As a particular case of Theorem \ref{mainresult}, we find that $(G;V_I;V_O)$ is identifiable if $D(W) = V$, that is, if $W$ is a zero forcing set in the graph $G$. This is the topic of the following theorem.  

\begin{theorem}
	\label{mainresult2}
	Let $G = (V,E)$ be an undirected graph and let $V_I, V_O \subseteq V$. If $V_I \cap V_O$ is a zero forcing set in $G$ then $(G;V_I;V_O)$ is identifiable.
\end{theorem}

\begin{remark}
For a graph $G = (V,E)$, checking whether a given subset is a zero forcing set in $G$ can be done in time complexity $\mathcal{O}(n^2)$, where $n = |V|$ (cf. \cite{Trefois2015}). Consequently, checking the condition of Theorem \ref{mainresult2} is still feasible for large-scale graphs. Although the focus of this paper is on the \emph{analysis} of identifiability, we remark that Theorem \ref{mainresult2} can also be used in the \emph{design} of sets $V_I$ and $V_O$ that ensure identifiability of $(G;V_I;V_O)$. Specifically, input and output sets with small cardinality are obtained by choosing $V_I = V_O$ as a \emph{minimum zero forcing set} in $G$. Minimum zero forcing sets are known for several types of graphs including path, cycle, and complete graphs, and for the class of tree graphs (see Section IV-B of \cite{Monshizadeh2014}). It was shown that finding a minimum zero forcing set in general graphs is NP-hard \cite{Aazami2008}. However, there also exist heuristic algorithms for finding (minimum) zero forcing sets. For instance, it can be shown that for any graph $G$, it is possible to find a zero forcing set of cardinality $n - \diam(G)$, where $\diam(G)$ denotes the diameter of $G$.
\end{remark}

\begin{remark}
\label{remarkdirected}
It is interesting to remark that Theorem \ref{mainresult2} implies that for many networks it is sufficient to excite and measure only a fraction of nodes (see, for instance, Example \ref{example1}). This is in contrast with the case of identifiability of dynamical networks with unknown graph structure, for which it was shown that all nodes need to be measured \cite{Pare2013}. Apart from the fact that we assume that the graph $G$ is known, the rather mild conditions of Theorem \ref{mainresult2} are also due to the fact that we consider \emph{undirected} graphs with state matrices that satisfy \emph{sign constraints}. In fact, in the case of \emph{directed} graphs it can be shown that the condition $V_I \cup V_O = V$ is \emph{necessary} for identifiability, i.e., each node of the graph needs to be an input or output node (or both). To see this, let $G_d$ be a directed graph, and define $\mathcal{Q}_p(G_d)$ analogous to the definition for undirected graphs (Section \ref{sectiongraphtheory}), with the distinction that $X \in \mathcal{Q}_p(G_d)$ is not necessarily symmetric. Assume that $V_I \cup V_O \neq V$. We partition $X \in \mathcal{Q}_p(G_d)$, and pick a nonsingular $S \in \mathbb{R}^{n \times n}$ as
\begin{equation*}
X = \begin{bmatrix}
X_{11} & X_{12} \\ X_{21} & X_{22}
\end{bmatrix}, \quad S = \begin{bmatrix}
I & 0 \\ 0 & \epsilon I
\end{bmatrix},
\end{equation*}
where the row block $\begin{bmatrix}
X_{21} & X_{22}
\end{bmatrix}$ corresponds to the nodes in $V \setminus (V_I \cup V_O)$. The partition of $S$ is compatible with the one of $X$, and $\epsilon$ is a positive real number, not equal to $1$. If $X_{12}$ and $X_{21}$ are not both zero matrices, then $\bar{X} := S^{-1} X S$ is contained in $\mathcal{Q}_p(G_d)$ and $\bar{X} \neq X$, but $(X,M,N)$ and $(\bar{X},M,N)$ have the same Markov parameters. That is, $(G;V_I;V_O)$ is not identifiable. If both $X_{12}$ and $X_{21}$ are zero, then the Markov parameters of $(X,M,N)$ are independent of $X_{22}$, hence, $(G;V_I;V_O)$ is also not identifiable. Therefore, for directed graphs the condition $V_I \cup V_O = V$ is necessary for identifiability. The above discussion also implies that $V_I \cup V_O = V$ is necessary for identifiability of \emph{undirected graphs} for which $X \in \mathcal{Q}(G)$ (i.e., for which $X$ does not necessarily satisfy the sign constraints). Indeed, this can be shown by the same arguments as above, using $\epsilon = -1$. We conclude that the conditions for identifiability become much more restrictive once we remove either the assumption on sign constraints or the assumption that the graph is undirected.
\end{remark}
 
\begin{example}
\label{example1}
In this example, we illustrate Theorem \ref{mainresult2}. Consider the tree graph $G = (V,E)$ of Figure \ref{fig:treegraph}.
\begin{figure}[h!]
	\centering
	\includegraphics[width=\linewidth]{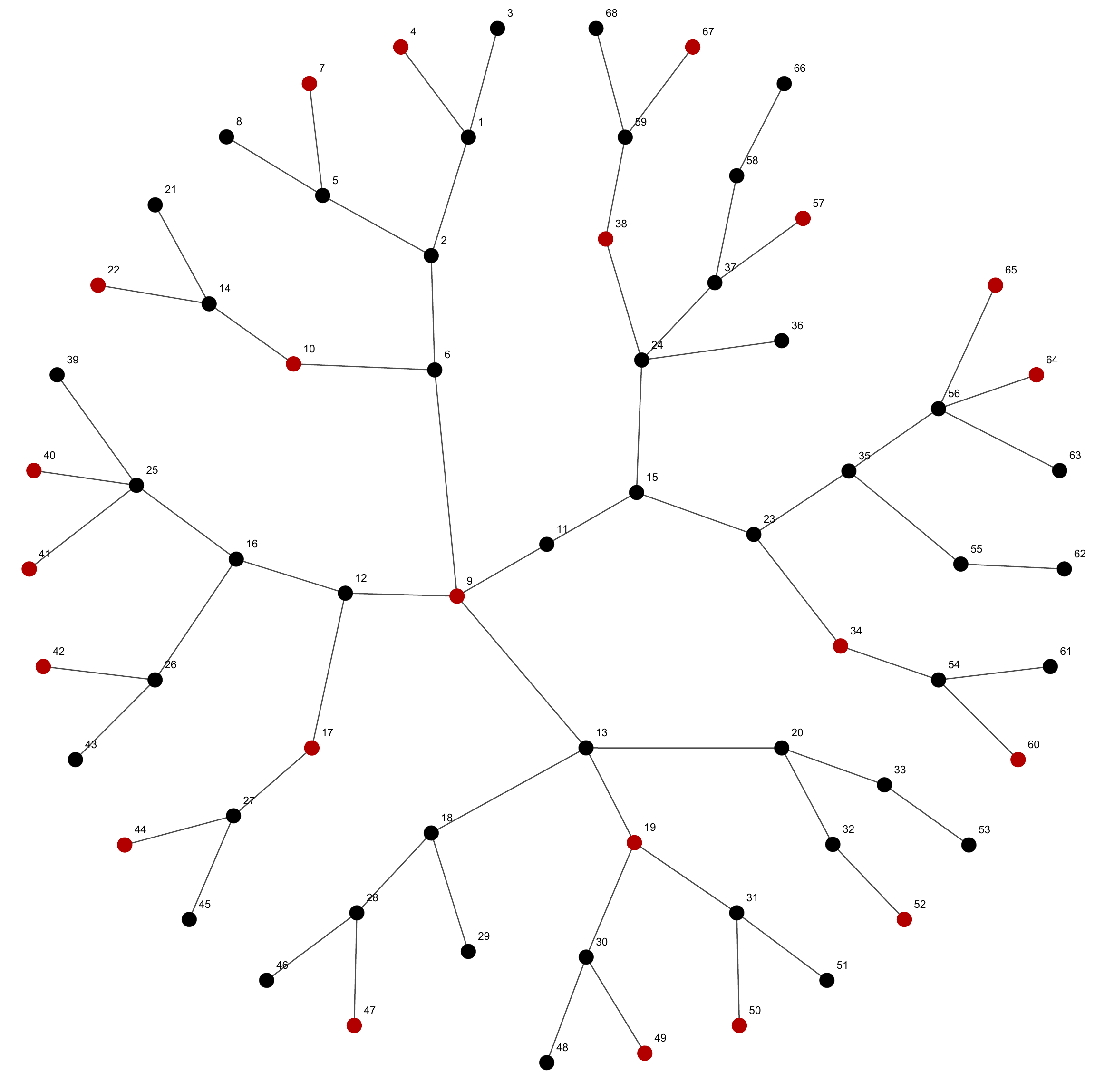}
	\caption{Tree graph $G$ with input and output set $V_I = V_O = \{4,7,9,10,\break17,19,22,34,38,40,41,42,44,47,49,50,52,57,60,64,65,67\}$.}
	\label{fig:treegraph}
\end{figure}
The input set $V_I$ and output set $V_O$ have been designed in such a way that $V_I = V_O$ is a minimum zero forcing set in $G$. In fact, the nodes of $V_I$ have been chosen as initial nodes of paths in a so-called \emph{minimal path cover} of $G$ \cite{Monshizadeh2014}, and therefore, $V_I$ is a minimum zero forcing set in $G$ by Proposition IV.12 of \cite{Monshizadeh2014}. This implies that also $V_I \cap V_O$ is a zero forcing set, and therefore, by Theorem \ref{mainresult2} we conclude that $(G;V_I;V_O)$ is identifiable.
\end{example}

\begin{example}
It is important to note that the condition in Theorem \ref{mainresult2} is not necessary for identifiability. This is shown in the following example. Consider a graph $G = (V,E)$, where $V = \{1,2,3\}$, and $E = \{(1,2),(2,1),(2,3),(3,2)\}$, and let $V_I = \{2\}$ and $V_O = V$. A straightforward calculation shows that any matrix $X \in \mathcal{Q}_p(G)$ can be identified from the Markov parameters $NXM$ and $N X^2 M$. This shows that $(G;V_I;V_O)$ is identifiable. However, note that $V_I \cap V_O = \{2\}$ is not a zero forcing set in $G$.
\end{example}

\section{Identifiability for higher-order dynamics}
\label{sectionextension}
The purpose of this section is to generalize the results of Section \ref{sectionmainresult} to the case of higher-order node dynamics. Specifically, suppose that node $i \in V$ has the associated dynamics
\begin{equation*}
\dot{x}_i(t) = \begin{cases}
A x_i(t) + B u_i(t) + E z_i(t) & \text{if } i \in V_I \\
A x_i(t) + E z_i(t) & \text{otherwise}
\end{cases},
\end{equation*}
where $x_i \in \mathbb{R}^q$ is the state of node $i$, $u_i \in \mathbb{R}^r$ is the input (only applied to nodes in $V_I$), and $z_i \in \mathbb{R}^s$ describes the coupling between the nodes. The real matrices $A,B,$ and $E$ are of appropriate dimensions. In addition to the above dynamics, we associate with each node $i \in V_O$ the output equation
\begin{equation*}
y_i(t) = C x_i(t),
\end{equation*} 
where $y_i \in \mathbb{R}^t$, and $C \in \mathbb{R}^{t \times q}$. The coupling variable $z_i$ is chosen as 
\begin{equation*}
z_i(t) = \sum_{j = 1}^n X_{ij} K x_j(t),
\end{equation*}
where $K \in \mathbb{R}^{s \times q}$, $X_{ii} \in \mathbb{R}$, $X_{ij} = X_{ji}$, and for $i \neq j$, $X_{ij} \geq 0$ and $X_{ij} > 0$ if and only if $(i,j) \in E$. We define $x := \col(x_1,x_2,\dots,x_n)$, $u := \col(u_{i_1},u_{i_2},\dots,u_{i_m})$, and $y := \col(y_{j_1},y_{j_2},\dots,y_{j_p})$, where $i_k \in V_I$ and $j_l \in V_O$ for all $k = 1,2,\dots,m$ and $l = 1,2,\dots,p$. Then, the dynamics of the entire network is described by the system
\begin{subequations}
\label{system2}
\begin{align}
\dot{x}(t) &= (I \otimes A + X \otimes EK) x(t) + (M \otimes B) u(t) \\
y(t) &= (N \otimes C) x(t),
\end{align}
\end{subequations} 
where the $(i,j)$-th entry of the matrix $X \in \mathcal{Q}_p(G)$ is equal to $X_{ij}$, and the matrices $M$ and $N$ are defined in \eqref{definitionMN}. Dynamics of the form \eqref{system2} arise, for example, when synchronizing networks of linear oscillators \cite{Scardovi2009}. In what follows, we use the notation $X_e := I \otimes A + X \otimes EK$, $M_e := M \otimes B$, and $N_e := N \otimes C$. 

We assume that the matrices $A,B,C,E$ and $K$ are known, and we are interested in conditions under which we can identify an induced subgraph $G_S$ of $G$. To make this precise, we say $(G_S;V_I;V_O)$ is \emph{identifiable with respect to \eqref{system2}} if for all $X, \bar{X} \in \mathcal{Q}_p(G)$ the following implication holds: 
\begin{equation*}
N_e (sI - X_e)^{-1} M_e = N_e (sI - \bar{X}_e)^{-1} M_e \implies X_S = \bar{X}_S,
\end{equation*}
where $X_e := I \otimes A + X \otimes EK$, $\bar{X}_e := I \otimes A + \bar{X} \otimes EK$, and the matrices $X_S, \bar{X}_S \in \mathcal{Q}_p(G)$ are the principal submatrices of $X$ and $\bar{X}$ corresponding to the nodes of $G_S$. The following theorem states conditions for identifiability of $(G_S;V_I;V_O)$ for the case of general network dynamics.

\begin{theorem}
	Let $G_S = (V_S,E_S)$ be an induced subgraph of the undirected graph $G = (V,E)$, and let $V_I,V_O \subseteq V$. Define $W := V_I \cap V_O$ and let $D(W)$ be the derived set of $W$ in $G$. Then $(G_S;V_I;V_O)$ is identifiable with respect to \eqref{system2} if $V_S \subseteq D(W)$ and $C (EK)^k B \neq 0$ for all $k \in \mathbb{N}$.
\end{theorem}

\begin{proof}
Consider two matrices $X, \bar{X} \in \mathcal{Q}_p(G)$ and define $X_e := I \otimes A + X \otimes EK$ and $\bar{X}_e := I \otimes A + \bar{X} \otimes EK$. Moreover, let $M_e := M \otimes B$, and $N_e := N \otimes C$. Suppose that $N_e X_e^k M_e = N_e \bar{X}_e^k M_e$ for all $k \in \mathbb{N}$. We want to prove by induction that $N X^k M = N \bar{X}^k M$ for all $k \in \mathbb{N}$. For $k = 1$, the equation $N_e X_e M_e = N_e \bar{X}_e M_e$ implies 
\begin{equation*}
(N \otimes C)(I \otimes A + X \otimes EK - I \otimes A - \bar{X} \otimes EK)(M \otimes B) = 0,
\end{equation*}
and hence $N(X - \bar{X})M \otimes CEKB = 0$. By assumption, $CEKB \neq 0$, and therefore $N X M = N \bar{X} M$. Next, suppose that $N X^i M = N\bar{X}^i M$ for all $i = 1,\dots,k$. The aim is to prove that $N X^{k+1} M = N\bar{X}^{k+1} M$. Note that we obtain
\begin{equation*}
N_e X_e^{k+1} M_e = NX^{k+1}M \otimes C(EK)^{k+1}B + \sum_{i=0}^k NX^{i}M \otimes R_i,
\end{equation*}
where $R_i$ is a matrix that depends on $A,B,C,E$ and $K$ only. Completely analogously, an expression for $N_e \bar{X}_e^{k+1} M_e$ can be derived. By the induction hypothesis, $NX^{i}M = N\bar{X}^{i}M$ for $i = 1,\dots,k$, and therefore $N_e X_e^{k+1} M_e = N_e \bar{X}_e^{k+1} M_e$ implies
\begin{equation*}
(NX^{k+1}M - N\bar{X}^{k+1}M) \otimes C(EK)^{k+1}B = 0.
\end{equation*}
Since $C(EK)^{k+1}B \neq 0$, we find $NX^{k+1}M = N\bar{X}^{k+1}M$. Therefore, $N X^k M = N \bar{X}^k M$ for all $k \in \mathbb{N}$. However, since $V_S \subseteq D(W)$ we find $X = \bar{X}$ by Theorem \ref{mainresult}. Hence $(G_S;V_I;V_O)$ is identifiable with respect to \eqref{system2}.
\end{proof}

\section{Conclusions}
\label{sectionconclusion}
In this paper we have considered the problem of identifiability of undirected dynamical networks. Specifically, we have assumed that the graph structure of the network is known, and we were interested in graph-theoretic conditions under which (a submatrix of) the network's state matrix can be identified. To this end, we have used a graph coloring rule called zero forcing. We have shown that a principal submatrix of the state matrix can be identified if the intersection of input and output nodes can color all nodes corresponding to the rows and columns of the submatrix. In particular, the entire state matrix can be identified if the intersection of input and output nodes constitutes a so-called zero forcing set in the graph. Checking whether a given set of nodes is a zero forcing set can be done in $\mathcal{O}(n^2)$, where $n$ is the number of nodes in the network \cite{Trefois2015}. We emphasize that the results we have presented here only treat the \emph{identifiability} of dynamical networks, and we have not discussed any network reconstruction algorithms, like in \cite{Shahrampour2015}, \cite{Materassi2012}, \cite{Hassan-Moghaddam2016}, \cite{vanWaarde2018}. However, if the conditions of Theorem \ref{mainresult2} are satisfied, then the state matrix of the network can be identified using any suitable method from system identification \cite{Ljung1999}.

\bibliographystyle{IEEEtran}
\bibliography{MyRef}

\end{document}